\documentclass[12pt]{article}

\usepackage{amsmath,amssymb,amsfonts,amsthm,epsfig,xcolor,epstopdf,graphicx,eucal,mathrsfs}
\usepackage{cite}
\usepackage{url}

\newtheorem{definition}{Definition}
\newtheorem{example}{Example}
\newtheorem{conjecture}{Conjecture}
\newtheorem{theorem}{Theorem}
\newtheorem{proposition}{Proposition}
\newtheorem{corollary}{Corollary}
\newtheorem{lemma}{Lemma}
\newtheorem{remark}{Remark}
\newtheorem*{challenge}{Challenge}

\title{The $3x+1$ Problem and Integer Representations}
\author{Jeffrey R. GOODWIN\\ \footnotesize\copyright Jeffrey R. GOODWIN April 3, 2012}

\begin{document}

\maketitle

\noindent{\it Dedicated to the memory of my mentor Dr. Marcia Jean Mason McKemie}

\begin{abstract}
The $3x+1$ Problem asks if whether for every natural number $n$, there exists a finite number of iterations of the piecewise function
$$
	f(2n)=n, \quad f(2n-1)=6n-2,
$$
with an iterate equal to the number $1$, or in other words, every sequence contains the trivial cycle $\left\langle {4,2,1}\right\rangle$. We use a set-theoretic approach to get representations of all inverse iterates of the number $1$. The representations, which are exponential Diophantine equations, help us study both the \textit{mixing} property of $f$ and the asymptotic behavior of sequences containing the trivial cycle. Another one of our original results is the new insight that the \textit{ones-ratio} approaches zero for such sequences, where the number of odd terms is \textit{arbitrarily large}.
\end{abstract}

\vskip.5cm

\section{Introduction.}\label{S.1}

The $3x+1$ Problem is a conjecture in additive number theory, and we credit it to Dr. Lothar Collatz from his student days in the 1930's, Collatz \cite{lC1986}. The vast body of literature on the conjecture shows an increasing rate of interest from scholars in a broad range of related disciplines, Lagarias \cite{jcL2010,jcL2011,jcL2012}.
\label{intro}
 	\begin{conjecture}\label{C.1}
 		(The $3x+1$ Problem) Let $x\in 
	         \mathbb{N}=\{1,2,3,\dots\}$. 
		Define the Collatz Function
		$f:\mathbb{N} \rightarrow \mathbb{N}$
		by
		$$
			 f(x)=
			\begin{cases}
				3x+1, &\textit{$x$ odd;}\\
				x/2, &\textit{$x$ even.}
			\end{cases}
		$$
		Show that every sequence $S=\left(f^{(0)}(x),f^{(1)}(x),f^{(2)}(x),\dots\right)$ contains the trivial cycle $\left\langle {4,2,1}\right\rangle$, where $f^{(0)}(x)=x$, or in other words, show that for each $x$ there exists $k\in \mathbb{N}$ such that $f^{(k)}(x)=1$.
	\end{conjecture}
Possible counterexamples:
\renewcommand{\labelenumi}{{\normalfont (\roman{enumi})}}
\begin{enumerate}
	\item \textit{Non-trivial cycle}, there exists $x$ for which $S$ is cyclic where, for all values of $k$, $f^{(k)}(x)\neq 1$.
	\item \textit{Divergent sequence}, there exists $x$ for which $\underset{k\rightarrow \infty}{\lim}f^{(k)}(x)=\infty$.
\end{enumerate}

An idealistic view that resolution of the conjecture would serve as a test of mankind's skill is fine, if it were to result in new math. A practical view is the idea that $f$ could someday augment the application of Benford's Law in fraud detection, Lagarias \cite{jcL2010,jcL2011,jcL2012}. Thus it is of interest to characterize the inverse iterates of the number $1$. Crandall \cite{rC1978} gives a formula representing all inverse iterates of $1$ as
\begin{equation}\label{E.1}
m=\left(2^{a_{k+1}}-\sum\nolimits_{i=0}^{k}2^{a_{i}}3^{k-i}\right) / 3^{k+1},
\end{equation}
where $m\in \mathbb{N},k\in \mathbb{N}_{0}:=\mathbb{N}\cup \left\{ 0\right\}$, and $0\leq a_{0}<a_{1}<\cdots <a_{k+1}$. 

B\"{o}hm and Sontacchi \cite{cBgS1978}, Wirsching \cite{gW1996}, Andrei, Kudlek and Niculescu \cite{sA2000} and Amig\'{o} \cite{jA2006}, all give the same formula by using different approaches. An artifact of those approaches is that there are non-integer solutions to the formula as well. In contrast, our approach amounts to a refinement of the formula giving just the natural number solutions. This refinement is in Section \ref{S.4}. In Section \ref{S.5} is our other new result that the \textit{ones-ratio} approaches zero for sequences containing the trivial cycle, where the number of odd terms is \textit{arbitrarily large}. In the next section we give a brief review of related work so the reader can visualize where our new results lay within the literature.

\section{Past results and related work.}\label{S.2}

In this section Proposition (\ref{P.1}) explains how (\ref{E.1}) relates to $3$\textit{-smooth numbers}. Propositions (\ref{P.2}) and (\ref{P.3}) do not explicitly appear in past results, we use them in the proofs of Theorems (\ref{T.2}) and (\ref{T.3}), and they help us understand (\ref{E.1}) in detail. By understanding (\ref{E.1}), we can examine Lemma (\ref{L.1}) in detail as it pertains to deriving rigorous estimates on the proportion of values $m$, having a given number of odd terms in the respective sequences, for which the $3x+1$ Problem holds below some bound. Such estimates are quite tricky and depend on the ratio of the number of even and odd terms in a given sequence, which we characterize in later sections. Afterwards, we review the approach that led to the main results of this paper followed by mention of some closely related approaches. We end this section with Proposition (\ref{P.6}) and Theorem (\ref{T.1}), which we make extensive use of in Sections \ref{S.3} and \ref{S.4}.

Blecksmith, McCallum and Selfridge \cite{rB1998} consider $3$\textit{-smooth numbers}, which are of the form $3^{k}2^{a}$ for some $k,a\in \mathbb{N}_{0}$. We credit $3$-smooth numbers to Ramanujan.
\begin{definition}\label{D.1}
Call a 3-smooth representation of \textit{$n$ special of level $k$} if it has the form $n=3^{k}2^{a_{0}}+3^{k-1}2^{a_{1}}+\cdots+3^{1}2^{a_{k-1}}+3^{0}2^{a_{k}}$, where $0=a_{0}<a_{1}<\cdots <a_{k}$.
\end{definition}
A result in Lagarias \cite{jcL1990} is that every natural number $n$ has at most one special representation (rep), abbreviated (reps) for the plural form, for each level $k$. For example, $n=19$ has special reps with $k=1,2$. Thus in Blecksmith, McCallum and Selfridge \cite{rB1998} is an equivalent formulation of rep (\ref{E.1}), as follows.
\begin{proposition}\label{P.1}
A natural number $x$ iterates to $1$ under the \textit{Collatz Function} if and only if there are natural numbers $a$ and $k$ such that $n=2^{a}-3^{k}x$ has a 3-smooth rep special of level $(k-1)$. The choice of $a$ and $k$ is not unique, if such a rep exists.
\end{proposition}
\begin{proof}
We use rep (\ref{E.1}) for a proof. For $x$ odd we have $a_{0}=0$, so $\sum\nolimits_{i=0}^{k}2^{a_{i}}3^{k-i}=n=2^{a_{k+1}}-3^{k+1}x$. Proposition \ref{P.2} completes the proof.
\end{proof}

\begin{proposition}\label{P.2}
Let $k\in \mathbb{N}_{0}$. Then 
$$\frac{2^{a_{k+1}+2}}{3^{k+2}}-\sum\limits_{i=0}^{k+1}\frac{2^{a_{i}}}{3^{i+1}}=\frac{2^{a_{k+1}}}{3^{k+1}}-\sum\limits_{i=0}^{k}\frac{2^{a_{i}}}{3^{i+1}}.
$$
\end{proposition} 
\noindent The proof is by induction and gives all values of $a$ and $k$. 
\begin{example}\label{X.1}We get a countably infinite number of algebraically equivalent reps for the number $3$, the first three are
$$
3\quad=\quad\frac{2^{5}}{3^{2}}-\frac{2^{1}}{3^{2}}-\frac{2^{0}}{3^{1}}\quad=\quad\frac{2^{7}}{3^{3}}-\frac{2^{5}}{3^{3}}-\frac{2^{1}}{3^{2}}-\frac{2^{0}}{3^{1}}\quad=\quad\frac{2^{9}}{3^{4}}-\frac{2^{7}}{3^{4}}-\frac{2^{5}}{3^{3}}-\frac{2^{1}}{3^{2}}-\frac{2^{0}}{3^{1}}.
$$
\end{example}
\begin{proposition}\label{P.3}
Minding Proposition (\ref{P.2}), for rep (\ref{E.1}) we have:
\end{proposition}
\begin{enumerate}
	\item \textit{The value $k+1$ gives the number of terms in the rep and the number of odd terms in $S$, including $m$ and the first occurrence of the number $1$.}
	\item \textit{The value $a_{k+1}$ gives the number of even terms in $S$.}
	\item \textit{The value $a_{i+1}-a_i$, $0 \leq i \leq k$, is the number of even terms between consecutive odd terms in $S$, e.g., $a_{k+1}-a_{k}$ is the number of even terms between the first occurrence of the number $1$ and the prior odd term in $S$.}
\end{enumerate}
\begin{proof} Rep (\ref{E.1}) gives every odd integer for which the $3x+1$ Problem holds. Without loss of generality, pick one such value of $m$ with the rep in least terms, recall Proposition (\ref{P.2}). We iterate $f$ to get
$$
f^{(1)}\left(\frac{2^{a_{k+1}}-\sum\nolimits_{i=0}^{k}2^{a_{i}}3^{k-i}}{3^{k+1}}\right)=\frac{2^{a_{k+1}}-\sum\nolimits_{i=1}^{k}2^{a_{i}}3^{k-i}}{3^{k}}.
$$
The result is even, and iterating $f$ (dividing by $2$) $a_{1}$ number of times gives the next odd integer in the sequence. We see, by inspection, that we prove all of the properties by induction.   
\end{proof}

Crandall \cite{rC1978} gives the following definition and lemma to help estimate the number of values $m$ for which the $3x+1$ Problem holds below some bound.
\begin{definition}\label{D.2}Denote by $G$ the set of infinite sequences $\left\{g_{k+1}, g_{k}, \dots, g_{1}\right\}$, where $a_{i+1}-a_{i}=g_{1}, \dots, a_{1}-a_{0}=g_{k+1}$ as defined in Proposition (\ref{P.3}), for odd $m$ only. Thus $g_{k+1}$ is the number of even terms between $m$ and the next odd term in $S$, and $g_{1}=2$ corresponds to the trivial cycle.
\end{definition}
\begin{lemma}\label{L.1}
For real number $r\in \mathbb{R}$, where $0<r$, the number of sequences in $G$ with $g_{1}+ \cdots + g_{k+1} \leq r$ is greater than or equal to $\bigl(2\lfloor(r-2)/6k\rfloor\bigr)^{k}$.
\end{lemma}
\begin{remark}\label{R.1}Each $m$ has a unique sequence in $G$ for some fixed $k$, see Crandall \cite{rC1978}, which is a basic fact and given in a different way in Amig\'{o} \cite{jA2006}. A consequence of this fact is that one must account for Proposition \ref{P.2} when deriving rigorous estimates. 
\end{remark}
\begin{example}\label{X.2}
Fix $m = 3$. Then $S=\left(3,10,5,16,8,4,2,1\right)$ and the corresponding sequence in $G$ is $\left(1,4\right)$, where $g_{2}=1$ and $g_{1}=4$, and we say that there are two odd terms in $S$ as we are not counting the number $1$.
\end{example}

Here we continue the approach of Goodwin \cite{jrG2003}, which uses patterns in the odd-inverse iterates of the number $1$ to create a partitioning scheme. A brief review follows and must be read carefully so that the reader will understand the new results in later sections. Noted in Crandall \cite{rC1978}, Cadogan \cite{cC1984} and Andrei and Masalagiu \cite{sA1998} is the trivial result 
\begin{equation}\label{E.2}
f^{(2b_{0}+3)}\Bigl\{\bigl(2^{2b_{0}+2}-1\bigr)/3\Bigr\} = 1,
\end{equation}
where $b_{0}\in N_{0}$. Our partitioning scheme is 
$$
\Bigl\{\bigl(2^{2b_{0}+2}-1\bigr)/3\Bigr\} =\left\{ \frac{2^{6b_{0}+6}-1}{3}\right\} \cup \left\{ \frac{2^{6b_{0}+4}-1}{3}\right\} \cup \left\{ \frac{2^{6b_{0}+2}-1}{3}\right\}.
$$ We ignore the set $\bigl\{(2^{6b_{0}+6}-1)/3\bigr\} $ as it is comprised of multiples of the number $3$, and odd-inverse iterates of multiples of $3$ do not exist, see Wirsching \cite{gW1996} or Amig\'{o} \cite{jA2006} for a proof. By working backwards we get the sets which map to the remaining sets in our partition, and we label each time we apply this approach by Level $n$. Thus in Goodwin \cite{jrG2003} we get the next proposition.

\begin{proposition}\label{P.4} Let $b_{i},i\in \mathbb{N}_{0}$.
\smallskip
 
\textbf{Level 1:}

$f^{(2b_{0}+3)}\left\{\frac{2^{2b_{0}+2}-1}{3}\right\} = 1$.

\smallskip

\textbf{Level 2:}

$f^{(2b_{0}+2)}X \subset \left\{ \frac{2^{2b_{0}+2}-1}{3}\right\}$, \quad $X=\left\{ 2^{2b_{0}}\bigl(\frac{2^{6b_{1}+5}-5}{9}\bigr) +\frac{2^{2b_{0}}-1}{3}\right\}$

$f^{(2b_{0}+3)}Y \subset \left\{ \frac{2^{2b_{0}+2}-1}{3}\right\}$, \quad $Y=\left\{ 2^{2b_{0}}\bigl(\frac{2^{6b_{1}+4}-7}{9}\bigr) +\frac{2^{2b_{0}}-1}{3}\right\}$.

\smallskip

\textbf{Level 3:}

$f^{(2b_{0}+2)}\left\{ 2^{2b_{0}}\left( 2^{6b_{1}}\bigl( \frac{2^{18b_{2}+12}-19}{27}\bigr) +5\bigl(\frac{2^{6b_{1}}-1}{9}\bigr) \right) +\frac{2^{2b_{0}}-1}{3}\right\} \subset X$

$f^{(2b_{0}+2)}\left\{ 2^{2b_{0}}\left( 2^{6b_{1}}\bigl( \frac{2^{18b_{2}+11}-23}{27}\bigr) +5\bigl(\frac{2^{6b_{1}}-1}{9}\bigr) \right) +\frac{2^{2b_{0}}-1}{3}\right\} \subset Y$

$f^{(2b_{0}+2)}\left\{ 2^{2b_{0}}\left( 2^{6b_{1}}\bigl( \frac{2^{18b_{2}+20}-31}{27}\bigr) +5\bigl(\frac{2^{6b_{1}}-1}{9}\bigr) \right) +\frac{2^{2b_{0}}-1}{3}\right\} \subset X$

$f^{(2b_{0}+2)}\left\{ 2^{2b_{0}}\left( 2^{6b_{1}}\bigl( \frac{2^{18b_{2}+7}-47}{27}\bigr) +5\bigl(\frac{2^{6b_{1}}-1}{9}\bigr) \right) +\frac{2^{2b_{0}}-1}{3}\right\} \subset Y$

$f^{(2b_{0}+2)}\left\{ 2^{2b_{0}}\left( 2^{6b_{1}}\bigl( \frac{2^{18b_{2}+10}-79}{27}\bigr) +5\bigl(\frac{2^{6b_{1}}-1}{9}\bigr) \right) +\frac{2^{2b_{0}}-1}{3}\right\} \subset X$

$f^{(2b_{0}+2)}\left\{ 2^{2b_{0}}\left( 2^{6b_{1}}\bigl( \frac{2^{18b_{2}+21}-143}{27}\bigr) +5\bigl(\frac{2^{6b_{1}}-1}{9}\bigr) \right) +\frac{2^{2b_{0}}-1}{3}\right\} \subset Y$

$f^{(2b_{0}+3)}\left\{ 2^{2b_{0}}\left( 2^{6b_{1}}\bigl( \frac{2^{18b_{2}+19}-29}{27}\bigr) +7\bigl(\frac{2^{6b_{1}}-1}{9}\bigr) \right) +\frac{2^{2b_{0}}-1}{3}\right\} \subset X$

$f^{(2b_{0}+3)}\left\{ 2^{2b_{0}}\left( 2^{6b_{1}}\bigl( \frac{2^{18b_{2}+6}-37}{27}\bigr) +7\bigl(\frac{2^{6b_{1}}-1}{9}\bigr) \right) +\frac{2^{2b_{0}}-1}{3}\right\} \subset Y$

$f^{(2b_{0}+3)}\left\{ 2^{2b_{0}}\left( 2^{6b_{1}}\bigl( \frac{2^{18b_{2}+9}-53}{27}\bigr) +7\bigl(\frac{2^{6b_{1}}-1}{9}\bigr) \right) +\frac{2^{2b_{0}}-1}{3}\right\} \subset X$

$f^{(2b_{0}+3)}\left\{ 2^{2b_{0}}\left( 2^{6b_{1}}\bigl( \frac{2^{18b_{2}+20}-85}{27}\bigr) +7\bigl(\frac{2^{6b_{1}}-1}{9}\bigr) \right)+\frac{2^{2b_{0}}-1}{3}\right\} \subset Y$

$f^{(2b_{0}+3)}\left\{ 2^{2b_{0}}\left( 2^{6b_{1}}\bigl( \frac{2^{18b_{2}+17}-149}{27}\bigr) +7\bigl(\frac{2^{6b_{1}}-1}{9}\bigr) \right) +\frac{2^{2b_{0}}-1}{3}\right\} \subset X$

$f^{(2b_{0}+3)}\left\{ 2^{2b_{0}}\left( 2^{6b_{1}}\bigl( \frac{2^{18b_{2}+16}-277}{27}\bigr) +7\bigl(\frac{2^{6b_{1}}-1}{9}\bigr)\right) +\frac{2^{2b_{0}}-1}{3}\right\} \subset Y$.
\end{proposition}

We get $216$ sets at Level $4$ and $11664$ sets at Level $5$, whereupon several patterns emerge:
\begin{enumerate}
	\item We get balanced sets in the sense that some parameters vary while others share the same geometric progressions.
	\item Excluding the set in Level $1$, half of the sets in Level $n$ take an odd number of iterations of $f$ to Level $(n-1)$, while the other half take an even number of iterations.
	\item As a specific example, we get balanced \textit{mixing} in the sense that half of the sets in Level $3$, where the number of iterations of $f$ is odd, map to $X$, while the other half map to $Y$.
	\item We partition the set in Level $1$ into three subsets, we ignore the subset comprised of multiples of $3$. Likewise, we partition $X$ and $Y$ into nine subsets each in Level $2$. Three of each of the nine subsets are comprised of multiples of $3$, so we ignore them as well. In Level $3$, we partition each set into twenty-seven subsets wherein nine subsets of each set are comprised of multiples of $3$ and so on. The pattern continues, and it should be clear how $12(27-9)=216$ sets in Level $4$.

We see that the set in Level $1$ maps to the number $1$, and that both sets in Level $2$ map to the set in Level $1$. But the pattern changes at Level $3$. Here we see that six of the sets in Level $3$ map to one of the sets in Level $2$, while the remaining six sets in Level $3$ map to the remaining set in Level $2$. The pattern to the mapping is $1,2\cdot3^{0},2\cdot3^{1},\dots,2\cdot3^{k}$. 
\end{enumerate} 

Our approach is tedious, however, there is an abstract version of it in Wirsching \cite{gW1996} which we give in the next lemma.
\begin{lemma}\label{L.2}
Let $m$ be an odd natural number for which the $3x+1$ Problem holds, where $m \not\equiv 0\pmod {3}$. The number of infinite sets per Level $n$ partitioning the inverse iterates of the number $m$, where the value $n$ counts the number of odd terms in sequences containing $m$ (or the trivial cycle in the case $m=1$), is given by 
$$
2^{n-1}3^{\bigl(n(n-3)+2\bigr)/2},
$$where $n\in \mathbb{N}$.
\end{lemma}The sets are residue classes modulo powers of $3$, and an advantage to our approach is that it lets us study the mixing by tracking just the smallest values in such sets, which we denote as follows.
  
\begin{definition}\label{D.3}
We get \textit{primitive seeds} by putting $b_{i}$ equal to zero in each set of every level as shown in Proposition (\ref{P.4}). 
\end{definition}
\begin{example}\label{X.3}
One of the $12$ primitive seeds in Level $3$ is $(2^{11}-23)/27$. So we may study the mixing property of its set by tracking just its value.
\end{example}

Cadogan \cite{cC1996} gives the next proposition, which helps us compare our approach with the one of Amig\'{o} \cite{jA2006}.
\begin{proposition}\label{P.5}
Let $(m_1,m_2,m_3,\dots)$ be a sequence of odd natural numbers such that $m_{i}=4m_{i-1}+1$, $i \geq 2$. Then $f^{(k)}(m_1)=1 \Longrightarrow                  
f^{(k+2i-2)}(m_i)=1$.
\end{proposition} For example, the $3x+1$ Problem holds for $m=3$, thus it holds for the set $\left\{3,13,53,...\right\}$. As in Amig\'{o} \cite{jA2006}, we define \textit{seeds} as the smallest members of such sets. Likewise, within Level $3$ is the expression $2^{6b_{1}}(2^{18b_{2}+11}-23)/27 +5(2^{6b_{1}}-1)/9$, which is just a collection of seeds with $(2^{11}-23)/27$ as the primitive seed. 

Andrei, Kudlek and Niculescu \cite{sA2000} and Andrei and Masalagiu \cite{sA1998} give various sets in the form of exponential Diophantine equations for which the $3x+1$ Problem holds, but their method does not partition the inverse iterates of the number $1$ in the same way as in Proposition (\ref{P.4}).

We end this section with a few more results of later use. Goodwin \cite{jrG2003} gives the next proposition.
\begin{proposition}\label{P.6}
Let $k,b\in \mathbb{N}_{0}$. Then 
$$
2^{3^{k+1}+k+2+2b3^{k+1}} \equiv 3^{k+2}-2^{k+2}\pmod{3^{k+2}},
$$ where
\begin{equation}\label{E.3}
3^{k+2}-2^{k+2}=\sum\nolimits_{j=0}^{k}3^{k-j}2^{j+1}+3^{k+1}.
\end{equation}
\end{proposition}

Finally, we state a well-known theorem in Rosen \cite{khRG2000}, see also DeLeon \cite{mjD1982, mjD1984}.
\begin{theorem}\label{T.1}
Let $n$ be a positive integer with a primitive root. If $k$ is a positive integer and $a$ is an integer relatively prime to $n$, then the congruence $x^k \equiv a\pmod {n}$ has a solution if and only if $a^{\phi\left(n\right)/d} \equiv 1\pmod {n}$, where $d=gcd\bigl(k,\phi\left(n\right)\bigr)$ and $\phi\left(n\right)$ is the Euler Phi-Function. If solutions to $x^k \equiv a\pmod {n}$ exist, then there are $d$ incongruent solutions modulo $n$.
\end{theorem}

\section{The two corner cases.}\label{S.3}

The even-odd nature of our partitioning scheme prompts us to formulate two reps,
	\begin{equation}\notag
          E=\left\{\frac{2^{e_{k}}-3^{k+1}-3^{k}2^{2b_{0}+1}-\sum\nolimits_{j=1}^{k}3^{k-j}2^{2b_{0}+\sum\nolimits_{i=1}^{j}2b_{i}3^{i}+\sum\nolimits_{i=1}^{j}\upsilon_{i}+1}}{3^{k+2}}\right\},
         \end{equation}
         \begin{equation}\notag
         O=\left\{\frac{2^{o_{k}}-3^{k+1}-3^{k}2^{2b_{0}+2}-\sum\nolimits_{j=1}^{k}3^{k-j}2^{2b_{0}+\sum\nolimits_{i=1}^{j}2b_{i}3^{i}+\sum\nolimits_{i=1}^{j}\upsilon_{i}+2}}{3^{k+2}}\right\}.
	\end{equation}

In this section we study a special corner case set from each of these reps, where the sets map to a subset of themselves by either an even or odd number of iterations of $f$ between levels, respectively. It is important that the reader understands the corner cases and their proofs before continuing to Section \ref{S.4}, where we discuss the even-odd reps $E$ and $O$. The next lemma gives us the corner cases.

\begin{lemma}\label{L.3}
Let $k,b_{i}\in \mathbb{N}_{0}$. The $3x+1$ Problem holds for sets
\begin{align}
	&\left\{\frac{2^{3^{k+1}+k+2+\sum\nolimits_{i=0}^{k+1}2b_{i}3^{i}}-\sum\nolimits_{j=0}^{k}3^{k-j}2^{j+1+\sum\nolimits_{i=0}^{j}2b_{i}3^{i}}-3^{k+1}}{3^{k+2}}\right\}, \label{E.4}\\
	&\left\{ \frac{2^{2\cdot3^{k+1}-2+\sum\nolimits_{i=0}^{k+1}2b_{i}3^{i}}-\sum\nolimits_{j=0}^{k}3^{k-j}2^{3^{j+1}-1+\sum\nolimits_{i=0}^{j}2b_{i}3^{i}}-3^{k+1}}{3^{k+2}}\right\}=A.\label{E.5}
\end{align}
\end{lemma}

\begin{proof} 
See Goodwin \cite{jrG2003} for a proof of (\ref{E.4}). Next we show that (\ref{E.5}) gives positive-odd integers. Fix $b_i=0$ and let the $3$-smooth (special of level $k$) part of $A$ equal $a$, so $a=\sum\nolimits_{j=0}^{k}3^{k-j}2^{3^{j+1}-1}+3^{k+1}$. Every power of $3$ has a primitive root and $\phi\left(3^{k+2}\right)=2\cdot3^{k+1}$. Thus by applying Theorem (\ref{T.1}), we must show that there exists a value of $c$ such that $a^{2\cdot3^{k+1}/d} \equiv 1\pmod{3^{k+2}}$, where $d=gcd\bigl(c+\sum\nolimits_{i=1}^{k}\upsilon_{i}+2,2\cdot3^{k+1}\bigr)$. We have $c=3^{k+1}-1$ and maximizing $\upsilon_{i}$, as defined in Section \ref{S.4}, we have $\sum\nolimits_{i=1}^{k}2\cdot3^{i}=3^{k+1}-3 \Longrightarrow c+\sum\nolimits_{i=1}^{k}\upsilon_{i}+2=2\cdot3^{k+1}-2$, and $d=2=\left(2\cdot3^{k+1}-2,2\cdot3^{k+1}\right) \Longrightarrow b=3^{k+1}$, where $b=2\cdot3^{k+1}/d$. Thus the congruence $a^b \equiv 1\pmod{3^{k+2}}$ holds as $gcd\left(a,3^{k+2}\right)=1$. Note that Proposition (\ref{P.6}) shows that $2^{2b_{k+1}3^{k+1}} \equiv 1\pmod{3^{k+2}}$. By using the expansion of the representation, see Proposition (\ref{P.4}) for an example, we conclude that the congruence gives positive-odd integers. Finally, we can apply Proposition (\ref{P.1}) as $A$ is of the form which iterates to $1$ under $f$; however, we give a proof by induction as in Goodwin \cite{jrG2003} to show how (\ref{E.5}) in Level $n$ maps to a subset of itself in Level $(n-1)$ all the way down to Level $3$, and to show why we call this set a corner case. Base case: Let $b_{0}=0$. Then

\begin{align} \notag
	 & f^{(3)}\left\{ \frac{2^{2\cdot3^{k+1}+4+\sum\nolimits_{i=1}^{k+1}2b_{i}3^{i}}-\sum\nolimits_{j=1}^{k}3^{k-j}2^{3^{j+1}-1+\sum\nolimits_{i=1}^{j}2b_{i}3^{i}}-2^23^{k}-3^{k+1}}{3^{k+2}}\right\}=\\ \notag
	&=f^{(2)}\left\{ \frac{2^{2\cdot3^{k+1}+4+\sum\nolimits_{i=1}^{k+1}2b_{i}3^{i}}-\sum\nolimits_{j=1}^{k}3^{k-j}2^{3^{j+1}-1+\sum\nolimits_{i=1}^{j}2b_{i}3^{i}}-2^23^{k}}{3^{k+1}}\right\}=\\ \notag
	&=\left\{ \frac{2^{2\cdot3^{k+1}+2+\sum\nolimits_{i=1}^{k+1}2b_{i}3^{i}}-\sum\nolimits_{j=1}^{k}3^{k-j}2^{3^{j+1}-3+\sum\nolimits_{i=1}^{j}2b_{i}3^{i}}-3^{k}}{3^{k+1}}\right\}=B.
\end{align}

\begingroup
\allowdisplaybreaks
\noindent Induction on $b_{0}$ gives
\begin{align} \notag
	& f^{(2b_{0}+3)}\left\{ \frac{2^{2\cdot3^{k+1}-2+\sum\nolimits_{i=0}^{k+1}2b_{i}3^{i}}-\sum\nolimits_{j=0}^{k}3^{k-j}2^{3^{j+1}-1+\sum\nolimits_{i=0}^{j}2b_{i}3^{i}}-3^{k+1}}{3^{k+2}}\right\}=\\ \notag
	&=f^{(2b_{0}+2)}\left\{ \frac{2^{2\cdot3^{k+1}-2+\sum\nolimits_{i=0}^{k+1}2b_{i}3^{i}}-\sum\nolimits_{j=0}^{k}3^{k-j}2^{3^{j+1}-1+\sum\nolimits_{i=0}^{j}2b_{i}3^{i}}}{3^{k+1}}\right\}=\\ \notag
	&=f^{(2)}\left\{ \frac{2^{2\cdot3^{k+1}+4+\sum\nolimits_{i=1}^{k+1}2b_{i}3^{i}}-\sum\nolimits_{j=1}^{k}3^{k-j}2^{3^{j+1}-1+\sum\nolimits_{i=1}^{j}2b_{i}3^{i}}-2^23^{k}}{3^{k+1}}\right\}=\\ \notag
	&=\left\{ \frac{2^{2\cdot3^{k+1}+2+\sum\nolimits_{i=1}^{k+1}2b_{i}3^{i}}-\sum\nolimits_{j=1}^{k}3^{k-j}2^{3^{j+1}-3+\sum\nolimits_{i=1}^{j}2b_{i}3^{i}}-3^{k}}{3^{k+1}}\right\}=B.
\end{align}
\endgroup
\noindent Define $A^*$ as $A$, but without the leftmost term, that is, subtract $1$ from $k$ to get 
$$
\left\{ \frac{2^{2\cdot3^{k}-2+\sum\nolimits_{i=0}^{k}2a_{i}3^{i}}-\sum\nolimits_{j=0}^{k-1}3^{k-1-j}2^{3^{j+1}-1+\sum\nolimits_{i=0}^{j}2a_{i}3^{i}}-3^{k}}{3^{k+1}}\right\}=A^*.
$$
\noindent We show $B \subset A^*$, in the following way
\begin{multline} \notag
	\frac{2^{2\cdot3^{k+1}+2+\sum\nolimits_{i=1}^{k+1}2b_{i}3^{i}}-\sum\nolimits_{j=1}^{k}3^{k-j}2^{3^{j+1}-3+\sum\nolimits_{i=1}^{j}2b_{i}3^{i}}-3^{k}}{3^{k+1}}=\\
	=\frac{2^{2\cdot3^{k}-2+\sum\nolimits_{i=0}^{k}2a_{i}3^{i}}-\sum\nolimits_{j=0}^{k-1}3^{k-1-j}2^{3^{j+1}-1+\sum\nolimits_{i=0}^{j}2a_{i}3^{i}}-3^{k}}{3^{k+1}}\\
	\Longleftrightarrow 2^{2\cdot3^{k}(3b_{k+1}+3)+2+\sum\nolimits_{i=1}^{k}2b_{i}3^{i}}-\sum\nolimits_{j=1}^{k}3^{k-j}2^{3^{j+1}-3+\sum\nolimits_{i=1}^{j}2b_{i}3^{i}}=\\
	=2^{2\cdot3^{k}(a_{k}+1)-2+\sum\nolimits_{i=0}^{k-1}2a_{i}3^{i}}-\sum\nolimits_{j=0}^{k-1}3^{k-1-j}2^{3^{j+1}-1+\sum\nolimits_{i=0}^{j}2a_{i}3^{i}},
\end{multline}
where $a_{i}=3b_{i+1}+2$. So it follows with $A^{**}$ defined as $A$, but with $k-2$, $f^{(2b_{0}+3)}A \subset A^*,f^{(2b_{0}+3)}A^* \subset A^{**}, \cdots$ all the way down to the set
$$
\Biggl\{ 2^{2b_{0}}\Biggl( 2^{6b_{1}}\biggl( \frac{2^{18b_{2}+16}-277}{27}\biggr) +7\biggl( \frac{2^{6b_{1}}-1}{9}\biggr)\Biggr) +\frac{2^{2b_{0}}-1}{3}\Biggr\}.
$$
\noindent Proposition (\ref{P.4}) completes the proof.
\end{proof}

\section{The even-odd representations.}\label{S.4}

\begin{theorem}\label{T.2}
	The set $E \cup O$ gives just the odd natural number solutions to rep (\ref{E.1}), as follows. Define a family of sets of consecutive natural numbers
	$$
	\left\{\bigl\{ 1,2,\dots,6\bigr\} ,\bigl\{ 1,\dots,18\bigr\} ,\dots,\bigl\{1,\dots,2\cdot3^{k^{\ast}}-1,2\cdot3^{k^{\ast }}\bigr\} \right\} _{k^{\ast}\in\left\{1,2,\dots,k-1,k\right\}},
	$$
	with variables $\upsilon_{1}\in \bigl\{ 1,2,\dots,6\bigr\} ,\dots,\upsilon_{j}\in \left\{ 1,\dots,2\cdot3^{k^{\ast }}-1,2\cdot3^{k^{\ast }}\right\}$. Denote
         $$
         e_{k}=c+2b_{0}+\sum\nolimits_{i=1}^{k+1}2b_{i}3^{i}+\sum\nolimits_{i=1}^{k}\upsilon_{i}+1
         $$ and
         $$
         o_{k}=c+2b_{0}+\sum\nolimits_{i=1}^{k+1}2b_{i}3^{i}+\sum\nolimits_{i=1}^{k}\upsilon_{i}+2,
         $$ and define
	\begin{equation}\label{E.6}
          E=\left\{\frac{2^{e_{k}}-3^{k+1}-3^{k}2^{2b_{0}+1}-\sum\nolimits_{j=1}^{k}3^{k-j}2^{2b_{0}+\sum\nolimits_{i=1}^{j}2b_{i}3^{i}+\sum\nolimits_{i=1}^{j}\upsilon_{i}+1}}{3^{k+2}}\right\},
         \end{equation}
         \begin{equation}\label{E.7}
         O=\left\{\frac{2^{o_{k}}-3^{k+1}-3^{k}2^{2b_{0}+2}-\sum\nolimits_{j=1}^{k}3^{k-j}2^{2b_{0}+\sum\nolimits_{i=1}^{j}2b_{i}3^{i}+\sum\nolimits_{i=1}^{j}\upsilon_{i}+2}}{3^{k+2}}\right\},
	\end{equation}
	where $c$ is the value of the form $c \equiv 2,4\pmod {6}$, with $2 \leq c< 2\cdot3^{k+1}$, per choice of $\upsilon_{i}$ giving the smallest positive integer solution, $k\in \mathbb{N}$ and $b_{i},b_{0}\in \mathbb{N}_{0}$.
\end{theorem}
\begin{proof} There are two parts to the proof. First we show that $E \cup O$ gives positive-odd integers, and that these integers iterate to $1$ under $f$. Let the $3$-smooth (special of level $k$) part of $E$ equal $a$ and let $b_0=b_i=0$, so $a=3^{k+1}+3^{k}2^{1}+\sum\nolimits_{j=1}^{k}3^{k-j}2^{\sum\nolimits_{i=1}^{j}\upsilon_{i}+1}$. Every power of $3$ has a primitive root and $\phi\left(3^{k+2}\right)=2\cdot3^{k+1}$. Thus by applying Theorem (\ref{T.1}), we must show that there exists a value of $c$ such that $a^{2\cdot3^{k+1}/d} \equiv 1\pmod{3^{k+2}}$, where $d=gcd\bigl(c+\sum\nolimits_{i=1}^{k}\upsilon_{i}+1,2\cdot3^{k+1}\bigr)$. As $a$ is invertible modulo $3^{k+2}$, there exists a unique value of $c$ for $\upsilon_{i}$ fixed in the given domain giving an integer value $b=2\cdot3^{k+1}/d$. Thus the congruence $a^b \equiv 1\pmod{3^{k+2}}$ holds as $gcd\left(a,3^{k+2}\right)=1$. Note that Proposition (\ref{P.6}) shows that $2^{2b_{k+1}3^{k+1}} \equiv 1\pmod{3^{k+2}}$. By using the expansion of the representation, see Proposition (\ref{P.4}) for an example, we conclude that the congruence gives positive-odd integers. Next we apply Proposition (\ref{P.1}) as the rep is of the form which iterates to $1$ under $f$, which we check as follows: 
\begin{multline}\notag
0<1+2b_{0}3^0<\upsilon_{1}+1+2b_{0}3^0+2b_{1}3^1<\cdots<\\[4pt] \notag
<\upsilon_{1}+\cdots+\upsilon_{k-1}+1+2b_{0}3^0+2b_{1}3^1+\cdots+2b_{k}3^{k}<\\[4pt]
<c+\upsilon_{1}+\cdots+\upsilon_{k}+1+2b_{0}3^0+2b_{1}3^1+\cdots+2b_{k+1}3^{k+1}.
\end{multline} The argument for $O$ is analogous, which completes the first part of the proof.

To complete the proof we must show that $E \cup O$ gives just the odd natural number solutions to rep (\ref{E.1}). Regarding the exponents of $2$ in $E$, we check that Proposition (\ref{P.3}) holds:
\begin{align}\notag
&\left(\upsilon_{1}+\cdots+\upsilon_{k}+1+2b_{0}3^0+2b_{1}3^1+\cdots+2b_{k}3^k+2b_{k+1}3^{k+1}+c\right)-\\[4pt] \notag
&-\left(\upsilon_{1}+\cdots+\upsilon_{k}+1+2b_{0}3^0+2b_{1}3^1+\cdots+2b_{k}3^k\right)=2b_{k+1}3^{k+1}+c\\[4pt] \notag
&\left(\upsilon_{1}+\cdots+\upsilon_{k}+1+2b_{0}3^0+2b_{1}3^1+\cdots+2b_{k}3^k\right)-\\[4pt] \notag
&-\left(\upsilon_{1}+\cdots+\upsilon_{k-1}+1+2b_{0}3^0+2b_{1}3^1+\cdots+2b_{k-1}3^{k-1}\right)=\upsilon _{k}+2b_{k}3^k\\ \notag
&\vdots\\ \notag
&\left(\upsilon_{1}+1+2b_{0}3^0+2b_{1}3^1\right)-\left(1+2b_{0}3^0\right)=\upsilon_{1}+2b_{1}3^1\\[4pt] \notag
&\left(1+2b_{0}3^0\right)-0=1+2b_{0}3^0. \notag
\end{align}
The differences sum to
\begin{equation}\label{E.8}
c+2b_{0}+\sum\nolimits_{i=1}^{k+1}2b_{i}3^{i}+\sum\nolimits_{i=1}^{k}\upsilon_{i}+1=e_{k}.
\end{equation} The argument for $O$ is analogous. We see for $n>2$ that the number of combinations of $\upsilon _{i}$ in $E \cup O$ gives the number of sets per Level $n$ as 
\begin{equation}\label{E.9}
2^{n-1}3^{\bigl(n(n-3)+2\bigr)/2}, 
\end{equation} where the formulation of $E \cup O$ begins at Level $3$. This offset (noted in Amig\'{o} \cite{jA2006} too) is unavoidable and it is okay because, by Proposition (\ref{P.2}), we have that Level $1$ is contained in Level $2$, Level $2$ is contained in Level $3$ and so on. The sets in Levels $1$ and $2$ from Proposition (\ref{P.4}) cover the cases $n=1,2$. Thus we get exactly the same number of sets per Level $n$ partitioning the odd-inverse iterates of the number $1$ as given by Lemma (\ref{L.2}), see Appendix \ref{A} for the details. Finally, it should be clear that the values $b_{i}$ and $b_{0}$ enumerate just the odd natural number solutions to rep (\ref{E.1}). 
\end{proof}

We can see how the sets in Level $(n+1)$ contain the sets in Level $n$, per Proposition \ref{P.2}, by looking at an example.
\begin{example}\label{X.4}Consider the following set from Level $2$
$$
\biggl\{2^{2b_{0}}\Bigl(2^{6b_{1}+4}-7\Bigr)/9+\Bigl(2^{2b_{0}}-1\Bigr)/3\biggr\}.
$$
Fixing $b_{1}=0$ gives
$$
\Bigl(16\cdot2^{2b_{0}}-7\cdot2^{2b_{0}} \Bigr)/9+\Bigl(2^{2b_{0}}-1\Bigr)/3=\Bigl(2^{2b_{0}+2}-1\Bigr)/3,
$$
which is exactly the set in Level $1$. Working backwards we consider the following set from Level $3$
$$
\biggl\{ 2^{2b_{0}}\Bigl( 2^{6b_{1}}\left(2^{18b_{2}+6}-37\right)/27+7\left(2^{6b_{1}}-1\right)/9\Bigr)+\left(2^{2b_{0}}-1\right)/3\biggr\}.
$$
Fixing $b_{2}=0$ gives the aforementioned set from Level $2$. We leave the general case to the reader. Therefore, from this fact and by Lemma (\ref{L.2}), the number of sets in Level $(n+1)$ which do not contain duplicate elements from Level $n$ is given by
\begin{equation}\label{E.10}
\left(2\cdot3^{n-1}-1\right)2^{n-1}3^{\bigl(n(n-3)+2\bigr)/2}.
\end{equation} However, the sets which do contain duplicate elements from Level $n$ also contain a countably infinite number of new elements. Moreover, every value of $c$, in the given domain for each level, gives a solution an equal number of times. For example, Table \ref{tab1} shows how each value of $c$ appears twice in Level $3$. Each value of $c$ appears $12$ times in Level $4$, and the formula for Level $n$ is
	\begin{equation}\label{E.11}
		\frac{2^{n-1}3^{\bigl(n(n-3)+2\bigr)/2}}{2\cdot3^{n-2}}=2^{n-2}3^{\bigl(n(n-5)+6\bigr)/2},
	\end{equation}
where $n \geq 2$. The proof follows directly from Lemma (\ref{L.2}) and the fact that there are $2\cdot3^{k}$ residue classes $m$ modulo $3^{k+1}$ per level with $m \not\equiv 0\pmod {3}$ in the \textquotedblleft pruned $3x+1$ Tree", see Lagarias \cite[p. 160]{jcL2010}. 
\end{example} 

Here is how the formulas work: fix $k$, choose values for the $\upsilon_{i}$ in either $E$ or $O$, set $b_{i}=b_{0}=0$ and calculate the value $c$ from the given domain, then variables $b_{i}$ and $b_{0}$ give an infinite set of odd natural numbers for which the $3x+1$ Problem holds, where the number of odd terms (not counting the number $1$) in the respective sequences is equal to the value $k+2$, minding the caveat $c=2$. Further, we have just shown that $c$ is the number of divide by $2$ steps between the first occurrence of $1$ and the prior odd term in $S$ for values of $m$ which are primitive seeds and that, together, the values $e_{k}$ and $o_{k}$ give the exact sum $r$ as required by Lemma (\ref{L.1}), again minding the caveat $c=2$.

\begin{definition}\label{D.4}Denote by $M_{k}$ the set given by $E \cup O$ joined with the sets in Levels $1$ and $2$ along with the elements given by multiplying each element of the aforementioned set by all powers of $2$, $\{2^1, 2^2, \dots\}$, where $m_{k} \in M_{k}$ and $k$ fixed denotes a set whose elements have $k$ number of odd terms in their respective sequences, that is, we omit the duplicate elements per level given by Proposition (\ref{P.2}).
\end{definition}
By corner case (\ref{E.4}) and Proposition (\ref{P.3}), or directly from Theorem (\ref{T.2}), we see that each $M_{k}$ is infinite. There are various proofs of this result in Crandall \cite{rC1978}, Wirsching \cite{gW1996}, Andrei, Kudlek, and Niculescu \cite{sA2000} and Shallit and Wilson \cite{jSdW1991}, and we state it formally as follows. 
\begin{corollary}\label{O.1}
For every fixed $k\in \mathbb{N}$ there exists a countably infinite number of $m_{k} \in M_{k}$ for which the $3x+1$ Problem holds, where $k$ counts the number of odd terms in the respective sequences.
\end{corollary}
Further, the set of pairwise distinct natural number solutions to rep (\ref{E.1}) is equivalent to $M_{k}$. Moreover, the sets $M_{k}$, where $k \geq 1$, are $2$\textit{-automatic} as defined and proved in Shallit and Wilson \cite{jSdW1991}.

\begin{remark}\label{R.2}Formula (\ref{E.9}) gives the number of different antisymmetric binary relations on a set of $n$ labeled points, see Pfeiffer \cite{gP2004} and Sloane \cite[A083667]{njaS2011}.
\end{remark}

\subsection{Some examples}\label{SUB.4}
Let $k\in \mathbb{N}_{0}$.
\begin{example}\label{X.5}
Consider $E$ and put $\upsilon_{1}=\cdots =\upsilon_{j}=1$ to get (\ref{E.4}).
\end{example}
\begin{example}\label{X.6}
Consider $O$ and put $\upsilon_{1}=6,\upsilon_{2}=18,\dots,\upsilon_{j}=2\cdot3^{k^{\ast }}$ to get (\ref{E.5}).
\end{example}
\begin{example}\label{X.7}
Consider $O$ and put $\upsilon_{1}=\cdots =\upsilon_{j}=1$ to get the following result conjectured in \cite{jrG2003}
\begin{equation}\label{E.12}
\left\{ \frac{2^{z_{k}+\sum\nolimits_{i=0}^{k+1}2b_{i}3^{i}}-\sum\nolimits_{j=0}^{k}3^{k-j}2^{j+2+\sum\nolimits_{i=0}^{j}2b_{i}3^{i}}-3^{k+1}}{3^{k+2}}\right\},
\end{equation}
\begin{equation}\label{E.13}
	\textit{where } z_{k}=
	\begin{cases}
		15\cdot3^{k-1}+k+3, &\textit{$k$ odd;}\\ 
		3^{k}+k+3, &\textit{$k$ even.}
	\end{cases}
\end{equation}
Notice we had to calculate enough initial values of the sequence
$$
(4,19,14,141,88,1223,738,10945,\dots),
$$
thus $9^{k}+2k+3=3^{2k}+2k+3,$ so that $3^{k}+k+3$ gives $z_{2i}=$$(4,14,88,\dots)$. Likewise, $15\cdot9^{k}+2k+4=15\cdot3^{2k}+2k+4$, so that $15\cdot3^{k-1}+k+3$ gives $z_{2i+1}=(19,141,1223,\dots)$. Calculations get large early in a numerical study, and it should be clear that we have shown the $3x+1$ Problem holds for the sets in these examples.
\end{example}

To demystify the formulation of $E \cup O$, we end this section with a conjecture equivalent to the $3x+1$ Problem.
\begin{conjecture}\label{C.2}
	Let $\omega,\beta \in \mathbb{N}_{0}$. Then the sets $E$ and $O$ partition all of the odds in $\mathbb{N}$ as follows
	\begin{align}
		&m\in \bigl\{2^{2\omega}\left(4\beta+3\right) + \left(2^{2\omega}-1\right)/3\bigr\} \Longleftrightarrow m\in E, \label{E.14}\\
		&m\in \bigl\{2^{2\omega}\left(8\beta+1\right) +  \left(2^{2\omega}-1\right)/3\bigr\} \Longleftrightarrow m\in O.\label{E.15}\\
		\notag
	\end{align}
	Note that $1\pmod 8 \cup 3\pmod 8 \cup 5\pmod 8 \cup 7\pmod 8$ are all odds in $\mathbb{N}$, $3\pmod 4 = 3\pmod 8 \cup 7\pmod 8$ and the formulation of $E$ and $O$ splits $5\pmod 8$ equally viz. $\{5,21,37,\dots\} \subset O$ and $\{13,29,45,\dots\} \subset E$.
\end{conjecture}

\section{Asymptotics of S containing the trivial cycle.}\label{S.5}

One way to study the asymptotic behavior of sequences containing the trivial cycle is to keep count of the number of even and odd terms. So disregarding the trivial cycle, define $e(m_{k})$ as the number of even terms in $S$ and $o(m_{k})$ as the number of odd terms, and we count $m$ but not the number $1$ in both definitions. Roosendaal \cite{eR2004} gives the next definition.
\begin{definition}\label{D.5}
The \textit{completeness} of $m_{k}$ is the value $C\left(m_{k}\right):=o\left(m_{k}\right)/e\left(m_{k}\right)$. 
\end{definition}
In Roosendaal \cite{eR2004} is a proof that $C\left(m_{k}\right) < \log 2/\log 3$. Using a graph-theoretic approach, the Andrei, Kudlek, and Niculescu \cite{sA2000} get the same result. Applegate and Lagarias \cite{dA2002} define essentially the same quantity as $C\left(m_{k}\right)$ by
$$
\rho\left(m_{k}\right):=o\left(m_{k}\right)/\sigma_{\infty}\left(m_{k}\right),
$$ and call it the \textit{ones-ratio}, where the \textit{total stopping time} $\sigma_{\infty}\left(m_{k}\right)$ is the total number of terms in the sequence given by $m_{k}$ (up to and including $1$, but not counting $m_{k}$). They, however, study a slightly different function. 
Let $t:\mathbb{N}\rightarrow \mathbb{N}$ be the \textit{Accelerated Collatz Function} defined by
$$
	t(x)=
	\begin{cases}
		(3x+1)/2, &\textit{$x$ odd;}\\
		x/2, &\textit{$x$ even.}
	\end{cases}
$$
\begin{example}\label{X.8}
Fix $m_{k} = 3$. Then under $t$ we have $S=\left(3,5,8,4,2,1\right)$, so $\rho \left(3\right)=2/5$. Under $f$ we have $S=\left(3,10,5,16,8,4,2,1\right)$, so $C\left(3\right)=2/5$.
\end{example}

Roosendaal \cite{eR2004} gives another way to study the asymptotic behavior of sequences containing the trivial cycle, as follows.
\begin{definition} \label{D.6}
Denote the quantity $\Gamma \left(m_{k}\right):=e\left(m_{k}\right)/\log m_{k}$, $m_{k}>1$.
\end{definition} 

The study in Roosendaal \cite{eR2004} suggests that $\Gamma\left(m_{k}\right)$ does not assume \textit{arbitrarily large} values and so makes the following conjecture.
\begin{conjecture}\label{C.3}
For all $m_{k} > 1$, there exists $C\left(m_{k}\right) < C_{\max} < \log 2/\log 3$. 
\end{conjecture}

Because Applegate and Lagarias \cite{dA2002} use $t$ instead of $f$, they define essentially the same quantity as $\Gamma \left(m_{k}\right)$ by
$$
\gamma\left(m_{k}\right):=\sigma_{\infty}\left(m_{k}\right)/\log m_{k},
$$ 
and call it the \textit{total stopping time ratio}. Thus we have $\Gamma \left(m_{k}\right)=\gamma\left(m_{k}\right)$. 
\begin{example}\label{X.9}
Under $t$ we have $\gamma \left(3\right)=5/\log 3$, while under $f$ we have $\Gamma \left(3\right)=5/\log 3$.
\end{example}

Applegate and Lagarias \cite{dA2002} state how trivially $\gamma\left(m_{k}\right)\geq 1/\log 2$ for all $m_{k}$ with equality for $m_{k}=2^{n}$, $n\in \mathbb{N}$. Further, they find a lower bound for the total stopping time ratio that holds for infinitely many $m_{k}$ (with an \textit{arbitrarily large} number of odd terms in the respective sequences, according to their proof method), and show that studying the size of the total stopping time ratio is the same as studying the size of the ones-ratio, as follows. For every sequence containing the trivial cycle we have
\begin{equation}\label{E.16}
\gamma\left(m_{k}\right)\geq\frac{1}{\log 2-\rho\left(m_{k}\right)\log 3}.
\end{equation}

Conjecture (\ref{C.3}) prompts us to search for a more insightful proof of the upper bound on $C\left(m_{k}\right)$, which we give here in the proof of the next theorem.
\begin{theorem}\label{T.3} Let $m_{k} \in  M_{k}$. Then 
$$
\underset{k\rightarrow \infty }{\lim }C\left(m_{k}\right)=\underset{k\rightarrow \infty }{\lim }\rho\left(m_{k}\right)=0.
$$
\end{theorem}
\begin{proof}By Proposition (\ref{P.3}), every odd $m_{k}$ is of the form $m_{k}=\left(2^{e\left(m_{k}\right)}-n\right)/3^{k}$, where $o\left(m_{k}\right)=k$. Since $2^{k \log 3/\log 2} = 3^{k}$, and from the positivity of $m_{k}$, we have that $e(m_{k}) > k\log 3/\log 2$. Therefore, for every odd $m_{k}$ there exists real number $\heartsuit\in \mathbb{R}$ such that $m_{k}=\left(2^{\heartsuit k\log 3/\log 2}-n\right)/3^{k}$. Since, by Proposition (\ref{P.1}), $n$ is a $3$-smooth rep special of level $(k-1)$, we have that the minimum value of $n$ increases with $k$. It follows that $\heartsuit$ increases with $k$ to ensure positivity of $m_{k}$. Hence as $k$ assumes \textit{arbitrarily large} values, per Corollary (\ref{O.1}), then so does $\heartsuit$. Another way to see this is from the corner cases, as they map to a subset of themselves per $k$, we are guarantied to always find new and unique values of $m_k$ as $k$ approaches infinity. This is important to note as there are an infinite number of $m_k$ for $k$ fixed. And we can ask if whether for some fixed $k$, we completely enumerate the odds in $\mathbb{N}$, with the answer being no. By Proposition (\ref{P.3}) we have $o\left(m_{k}\right)=k$ and $e\left(m_{k}\right)=\heartsuit k\log 3/\log 2$, which gives $C\left(m_{k}\right)=k / \left(\heartsuit k\log 3/\log 2\right)$. Putting $\heartsuit=1$ gives us the upper bound $C\left(m_{k}\right) < \log 2/\log 3$ because, via cancellation, $k/\left(k\log 3/\log 2\right)=\log 2/\log 3$. Therefore this upper bound holds for \textit{arbitrarily large} $k$, showing us that the corresponding increase of $\heartsuit$ forces $C\left(m_{k}\right)$ towards zero. From the definition of completeness, this result holds for every even $m_{k}$ too.
\end{proof}

The power of this theorem is that it forces us to take another look at how we apply probabilistic arguments to the parity sequences. It also says that Conjecture (\ref{C.3}) is true if Conjecture (\ref{C.1}) is true. Specifically, it says $C_{\max}$ exists for $M_{k}$. Furthermore, the upper bound on $C\left(m_{k}\right)$ is a corollary, and just like in Roosendaal \cite{eR2004} we can prove it in a more constructive way as follows.
\begin{corollary}\label{O.2} Let $m_{k} \in M_{k}$. Then
$$
C\left(m_{k}\right)<\underset{o(m_{k})\rightarrow \infty }{\lim }{\frac{o(m_{k})}{\left\lceil \log_2 \left(3^{o(m_{k})+1}-2^{o(m_{k})+1}\right) \right\rceil }}= \frac{\log 2}{\log 3}.
$$
\end{corollary} 
\begin{proof}
As $a_{k+1}$ gives the number of even terms in $S$, by Proposition (\ref{P.3}), we can get the upper bound for $C\left(m_{k}\right)$ as follows. Define the lower bound $a_{\min}$ for $a_{k+1}$ by noticing that the biggest power of $2$ must give a positive sum when the other exponents are in lowest terms. We have
$$
\sum\nolimits_{i=0}^{k}2^{i}/3^{i+1}=\left(3^{k+1}-2^{k+1}\right)/3^{k+1}=1-2^{k+1}/3^{k+1}. 
$$
We want integer $a_{\min }$ such that
\begin{align} \notag
& 2^{a_{\min }}/3^{k+1}-1+2^{k+1}/3^{k+1}>0\Longrightarrow 2^{a_{\min}}>3^{k+1}-2^{k+1}\Longrightarrow\\
&\Longrightarrow a_{\min }>\log _{2}\left(3^{k+1}-2^{k+1}\right) \Longrightarrow a_{\min }=\left\lceil \log _{2}\left(3^{k+1}-2^{k+1}\right) \right\rceil. 
\end{align}
\end{proof}

Corner case (\ref{E.4}) is rather special because it lets us prove Theorem (\ref{T.3}) in the other direction, for just this one case. To simplify the visual we reduce (\ref{E.4}) to just its primitive seeds and use Identity (\ref{E.3}) to get
$$
m^{*}_{k}=\left(2^{3^{k+1}+k+2}-3^{k+2}+2^{k+2}\right)/3^{k+2},
$$ where, by Proposition (\ref{P.3}), $e\left(m^{*}_{k}\right)=3^{k+1}+k+2$. Thus we have
\begin{equation}\label{E.17}
\underset{k\rightarrow \infty }{\lim }\Gamma\left(m^{*}_{k}\right)=\underset{k\rightarrow \infty }{\lim }\gamma\left(m^{*}_{k}\right)=\underset{k\rightarrow \infty }{\lim }\frac{3^{k+1}+k+2}{\log\left(\frac{2^{3^{k+1}+k+2}-3^{k+2}+2^{k+2}}{3^{k+2}}\right) }=\frac{1}{\log 2}.
\end{equation}
It follows that (\ref{E.4}) is an infinite set of odd natural numbers having a total stopping time ratio \textit{asymptotically equivalent} to the trivial case $m_{k}=2^{n}$, where the number of odd terms in the respective sequences is \textit{arbitrarily large}. Therefore, by Equation (\ref{E.16}), we have that $C\left(m^{*}_{k}\right)=\rho\left(m^{*}_{k}\right) \rightarrow 0$ as $k \rightarrow \infty$ for (\ref{E.4}). Furthermore, it should be clear that (\ref{E.4}) denotes an entire residue class in every Level $n$. By \textit{asymptotically equivalent} and \textit{arbitrarily large} we mean:

\begin{enumerate}
	\item \textit{A fixed value of $k=o\left(m_{k}\right)$ denotes a finite number of odd terms in the sequence generated by $m^{*}_{k}$. However, the set $m^{*}_{k}$ is infinite because $k\in \mathbb{N}$, hence the number of odd terms in the respective sequences is arbitrarily large.}
	\item \textit{A casual reader can use a software application such as Mathematica to check the monotonic convergence of Equation (\ref{E.17}) to the value $1/\log 2$.}
	\item \textit{For every fixed value of $k$ we have $\Gamma\left(m^{*}_{k}\right)=\gamma\left(m^{*}_{k}\right) > 1/\log 2$.}
	\item \textit{The limit value we get for Equation (\ref{E.17}) is valid, so we say that it is asymptotically equivalent to the trivial case $m_{k}=2^{n}$.}
\end{enumerate}

\begin{example} Consider the infinite set given by $2^{\lambda}-1$, where $\lambda \in \mathbb{N}$. For $\lambda=1$ we have $\left(1\right)$. For $\lambda=2$ we have $\left(\mathbf{3},10,\mathbf{5},16,8,4,2,1\right)$. For $\lambda=3$ we have
$\left(\mathbf{7},22,\mathbf{11},34,\mathbf{17},52,26,\mathbf{13},40,20,10,\mathbf{5},16,8,4,2,1\right).$ Thus $\lambda$ does not count the number of odds in the corresponding sequences as far as we can see in the finite case. However, the number of odds across all of the corresponding sequences is arbitrarily large (unbounded) for the set $2^{\lambda}-1$; recall that $M_{k}$ is infinite for $k$ fixed. Further, it is not known whether the $3x+1$ Problem holds for the set $2^{\lambda}-1$ nor for even an infinite subset of it. So the strength of the approach taken herein should be clear.
\end{example}

\begin{remark}\label{R.3} Let $m_{k} \in M_{k}$. Then from Theorem (\ref{T.3}) and Equation (\ref{E.16}) we have
$$
\underset{k\rightarrow \infty }{\lim }\Gamma\left(m_{k}\right)=\underset{k\rightarrow \infty }{\lim }\gamma\left(m_{k}\right)\geq1/\log 2.
$$
Nonetheless, assuming the result of Theorem ($1.1$) in Applegate and Lagarias \cite{dA2002} claiming the existence of a lower bound $> 14/29$ for the ones-ratio of some infinite set (where the number of odd terms in the respective sequences is arbitrarily large, i.e., unbounded) is likewise sound, we have found a non-constructive proof that the $3x+1$ Problem is false. 
\end{remark}

\begin{remark}\label{R.4} Theorem (\ref{T.3}) can be generalized, and it can be modified to give us a lower bound on the size of the smallest value of a non-trivial cycle relative to the number of odd terms in said cycle, should one exist. We can prove that the smallest value in a non-trivial cycle containing $k$ odd terms must be greater than $3^{k-1}$.   
\end{remark}

\begin{definition}\label{D.7}
The quantity $\Gamma\left(n\right)$ is a record if for all $m<n$, we have $\Gamma(m)<\Gamma(n)$. 
\end{definition}

In Roosendaal \cite{eR2004} is a list of $\Gamma\left(n\right)$ records. We know, from Proposition \ref{P.3}, that these records come from the list given by the smallest values $m$ such that the resulting sequence contains exactly $k$ odd terms, see Sloane \cite[A092893]{njaS2011}. Also, $c=4$ in the sequences associated with these records, and we want to know if this is always the case. Clearly the members of this sequence come from the set of primitive seeds, and an explicit formula for this sequence would offer a means to measure the growth rate of the $3x+1$ Tree.

In Roosendaal \cite{eR2004} is a list of completeness records with an analogous definition. The smallest record is  $C\left(3\right)=0.40$, while the largest is enormous 
$$
C\left(7\text{ }219\text{ }136\text{ }416\text{ }377\text{ }236\text{ }271\text{ }195\right)=0.606061.
$$So it seems prudent to make note of a simple fact
\begin{equation}\label{E:18}
\left(2^k\right)^{1-(\log2/\log6)}=\left(3^k\right)^{\log2/\log6},
\end{equation}for all $k \geq 1$, where we have
$$
\log2/\log6=0.386852<0.40; \quad 1-(\log2/\log6)=0.613147>0.606061.
$$

We close by mentioning that by combining Proposition (\ref{P.1}) and Proposition (\ref{P.3}) we get the following quantity defined in Roosendaal \cite{eR2004} 
$$
Res(m):=\frac{2^{e(m)}}{m3^{o(m)}}=\frac{1}{1-(n/2^{e(m)})}>1,
$$
where $Res(993)=1.253142$ is the largest value for all $m<2^{32}$. Andrei, Kudlek and Niculescu \cite{sA2000} and Monks \cite{kgM2002} consider the inverse of $Res(m)$.
\begin{remark}\label{R.5}
Corner case (\ref{E.4}) says that $Res(m_{k})$ gets arbitrarily close to $1$. 
\end{remark}

\section{The mixing property and future research.}\label{S.6}

Here we offer the reader a starting point for further research. The computational difficulty of calculating $c$ is related to the \textit{discrete logarithm problem}, see Rosen \cite{khRG2000}. We are therefore motivated to consider a recursive definition of the mixing property:
\begin{enumerate}
	\item Let the primitive seeds in Level $3$ constitute our finite list of starting sets.
	\item Next we define the operations that generate the sets in Level $4$ from the set of primitive seeds in Level $3$, and in general the sets in Level $(n+1)$ from the primitive seeds in Level $n$ for $n \geq 3$.
	\item Knowing that the rep for a given value $m$ is unique per level and by Proposition (\ref{P.2}), we have that reps (\ref{E.6}) and (\ref{E.7}), with $n$ arbitrarily large, are the smallest sets containing the starting sets and all others that we can derive from them by iteration of the generating operations.
\end{enumerate}
Next we gather our observations regarding the mixing and, considering Proposition \ref{P.3}, introduce a notation for the primitive seeds
$$
m_{k}:=(c_{k},\upsilon_{k},\upsilon_{k-1},\dots,\upsilon_{1},\upsilon_{0}), 
$$
where $\upsilon_{0}:=1$ if the primitive seed is in $E$, and defined as $2$ if it is in $O$, see Table \ref{tab1} for an example which also serves as Part (i) of the recursive definition. The subscripts on $c, E$ and $O$ are necessary, where Level $n=3$ correlates to a value of $k=1$.
\begin{remark}\label{R.6} The value $\upsilon_{0}$ is, by construction of the reps, the value which determines whether the number of iterations of $f$ are odd or even to the next odd term in $S$, following $m$.
\end{remark}

For reference, we give just the primitive seeds in $E \cup O$ as
$$
m_{k}=\left(2^{c_{k}+\sum\nolimits_{i=1}^{k}\upsilon_{i}+\upsilon_{0}}-3^{k+1}-3^{k}2^{\upsilon_{0}}-\sum\nolimits_{j=1}^{k}3^{k-j}2^{\sum\nolimits_{i=1}^{j}\upsilon_{i}+\upsilon_{0}}\right)/3^{k+2},
$$
where $\upsilon_{0} \in \left\{1,2\right\}$.

\begin{table}[ht]
	\renewcommand{\arraystretch}{1.5}
		\begin{center}
		\begin{tabular}{ | l | l | l | l | }\hline
			\multicolumn{2}{ | c | }{$E_{1}$} & \multicolumn{2}{ | c | }{$O_{1}$}\\ \hline
			$\mathbf{\left(2^{12}-19\right)/27}$ & $\mathbf{\left(10,1,1\right)}$ & $\left(2^{19}-29\right)/27$ & $\left(16,1,2\right)$\\ \hline
			$\left(2^{11}-23\right)/27$ & $\left(8,2,1\right) $ & $\left(2^{6}-37\right)/27$ & $\left(2,2,2\right)$\\ \hline
			$\left(2^{20}-31\right)/27$ & $\left(16,3,1\right) $ & $\left(2^{9}-53\right)/27$ & $\left(4,3,2\right)$\\ \hline
			$\left(2^{7}-47\right)/27$ & $\left(2,4,1\right) $ & $\left(2^{20}-85\right)/27$ & $\left(14,4,2\right)$\\ \hline
			$\left(2^{10}-79\right)/27$ & $\left(4,5,1\right) $ & $\left(2^{17}-149\right)/27$ & $\left(10,5,2\right)$\\ \hline
			$\left(2^{21}-143\right)/27$ & $\left(14,6,1\right) $ & $\mathbf{\left(2^{16}-277\right)/27}$ & $\mathbf{\left(8,6,2\right)}$\\ \hline
		\end{tabular}
		\end{center}		
		\caption{Starting sets in Level $3$ as given by the primitive seeds.}\label{tab1}
\end{table}

\newpage
As for our list of observations regarding the mixing we have:
\begin{enumerate}
	\item Half of the sets take an odd number of $k$ to reach the next odd term in $S$ while the others take an even number of $k$, the mixing appears balanced, and from calculation it seems sufficient to use the parity of the values of $\upsilon_{1}$ to deduce the mapping. We formalize these observations in the following 
	way
	\begin{multline}
		f^{\left( 2b_{0}+2\right) }E_{k}^{\upsilon_{1}}\subset O_{k-1}, \quad \text{if } \upsilon_{1} \text{ even;}\\
		\shoveleft f^{\left(2b_{0}+2\right)}E_{k}^{\upsilon_{1}}\subset E_{k-1}, \quad \text{if } \upsilon_{1} \text{ odd;}\\
		\shoveleft f^{\left(2b_{0}+3\right)}O_{k}^{\upsilon_{1}}\subset O_{k-1}, \quad \text{if } \upsilon_{1} \text{ even;}\\
		\shoveleft f^{\left(2b_{0}+3\right)}O_{k}^{\upsilon_{1}}\subset E_{k-1}, \quad \text{if } \upsilon_{1} \text{ odd,}\\ \notag
	\end{multline}
	where none of the iterates map to multiples of $3$ in $E_{k-1} \cup O_{k-1}$.
	\item Formula (\ref{E.11}) gives the number of times each value of $c_{k}$ gives a solution in a given level.
\end{enumerate}

The operations will be based on the two primitive seeds per level that we can predict, namely those from the corner cases. See the primitive seeds in bold text in Table \ref{tab1} as an example. 
\begin{remark}\label{R.7} We can, of course, do more calculations to predict more primitive seeds as in Example (\ref{X.7}).
\end{remark}
We end this section by emphasizing its goal, which is to finish the recursive definition, validate it with $f$, and use it to glean more information on the extremal values of the set of sequences which contain the trivial cycle. We leave it to the reader to finish this investigation.

\section{Conclusion.}\label{S.7}

With our approach we get just the natural number solutions to rep (\ref{E.1}), an incremental improvement which took over thirty years to achieve. Along the way we found that the ones-ratio approaches zero for sequences containing the trivial cycle, where the number of odd terms is arbitrarily large. Furthermore, we get an explicit form for the $3$-smooth reps within rep (\ref{E.1}), and for what it is worth, the past decade saw much research in the application of $3$-smooth numbers to the Double Base Number System, see  Mishra and Dimitrov \cite{pkM2008}. In addition, we have a slight similarity to the results in Gy\"{o}ry and Smyth \cite{kGcS2010} as the terms in $e_{k}$ and $o_{k}$ involve powers of $3$.

The author is of the opinion that reps (\ref{E.6}) and (\ref{E.7}) are a doorway to a deeper understanding of additive number theory. Many years ago the author cold called Prof Dr. Richard G. R. Pinch stating how the set in (\ref{E.2}) contained many pseudo-primes. Dr. Pinch replied within a day or so with a proof that $m$ is a pseudo-prime to the base $2^{b_{0}-1}$ for $b_0$ a prime. Many other mathematical results relate to this sequence, see Sloane \cite[A002450]{njaS2011}. Now we have reps (\ref{E.6}) and (\ref{E.7}) as a refinement of rep (\ref{E.1}), and so perhaps we can find more results like these.

\section*{Acknowledgements.}
The author thanks his wife Megan and their sons Raaf and Dane and daughter Sophie for allowing him the time to finish this investigation of the problem. The author thanks Prof. Dr. \c{S}tefan Andrei, Prof. Dr. Gabriel Ciobanu, and Olivier Rozier whose comments have improved the paper, and is grateful to Prof. Dr. Jeffrey C. Lagarias for past discussions which have helped the author understand the literature. A special thanks goes to the computer hardware and software which enabled the author to pattern search the inverse iterates of the number $1$.

\section{Appendix.}\label{A}
Wirsching \cite{gW1996} constructs a reversed form of rep (\ref{E.1}) under $t$ (we define $t$ in Section \ref{S.5}). We reproduce the reversed form here with the following definition and proposition for comparison with $E \cup O$.
\begin{definition}\label{D.8}
Denote by $F$ the set of finite sequences $\bigl\{(\alpha_{0},\dots,\alpha_{\mu})\bigr\}$, where\newline $\mu,\alpha_{0},\dots,\alpha_{\mu} \in \mathbb{N}_{0}$, and define for $s=(\alpha_{0},\dots,\alpha_{\mu})$ the following:
\renewcommand{\labelenumi}{{\normalfont (\roman{enumi})}}
\begin{enumerate}
	\item Length, $\ell (s):=\mu$
	\item Absolute value, $\left\vert s\right\vert :=\alpha_{0}+\cdots+\alpha_{\mu }$
         \item Norm, $\left\Vert s\right\Vert :=\left\vert s\right\vert +\ell (s)$
	\item For rational number $q \in \mathbb{Q}$ define $\zeta_{+}(q):=2q$
	\item For $q \in \mathbb{Q}$ define $\zeta_{-}(q):=(2q/3)-(1/3)$
	\item For $s \in F$ define $\zeta_{s}:=\zeta_{+}^{\alpha_{0}}\circ (\zeta_{-}\circ\zeta_{+}^{\alpha_{1}}) \circ \cdots \circ (\zeta_{-}\circ\zeta_{+}^{\alpha_{\mu}})$.
\end{enumerate} 
\end{definition}
\begin{proposition}\label{P.7}
Let $q \in \mathbb{Q}$ and $s=(\alpha_{0},\dots,\alpha_{\mu}) \in F$. Then 
$$
\zeta_{s}(q):=h(s)q-l(s),
$$
where
$$
h(s):=2^{\left\Vert s\right\Vert } / 3^{\ell (s)}, \quad l(s):=\sum\limits_{k=0}^{\mu -1}2^{k+\alpha _{0}+\cdots+\alpha _{k}} / 3^{k+1}.
$$
\end{proposition}
\begin{example}\label{X.10}
Let $m=3$. Then under $t$ we have 
$$
S=(3,5,8,4,2,1),
$$and by Definition \ref{D.8} we have
$$
s=(0,0,3); \quad \ell (s)=2; \quad \left\vert s\right\vert =3; \quad \left\Vert s\right\Vert =5; \quad q=1.
$$ 
Thus by Proposition \ref{P.7} we have
$$
\zeta_{s}(1)=\frac{2^{5}(1)}{3^{2}}-\frac{2^{0+0}}{3^{0+1}}-\frac{2^{1+0+0}}{3^{1+1}}=3,
$$ 
and we see that this rep is exactly the rep of $3$ in least terms as noted in Example \ref{E.1}. We call $s=(0,0,3)$ an admissible sequence as defined in Wirsching \cite{gW1996}. We can construct another admissible sequence by the concatenation $s\cdot w=(0,0,3+0,1)=(0,0,3,1)=s^{\ast }$, which corresponds to the inclusion of the first occurrence of the trivial cycle  
$$
S^{\ast}=(3,5,8,4,2,1,2,1).
$$
By Definition \ref{D.8} we have
$$
s^{\ast}=(0,0,3,1); \quad  \ell (s^{\ast})=3; \quad \left\vert s^{\ast}\right\vert =4; \quad \left\Vert s^{\ast}\right\Vert =7; \quad q=1.
$$
Thus by Proposition \ref{P.7} we have
$$
\zeta_{s^{\ast}}(1)=\frac{2^{7}(1)}{3^{3}}-\frac{2^{0+0}}{3^{0+1}}-\frac{2^{1+0+0}}{3^{1+1}}-\frac{2^{2+0+0+3}}{3^{2+1}}=3,
$$ 
which is equivalent to the next rep of $3$ in its number of terms and values of exponents, as given by Proposition (\ref{P.2}), and is noted in Example (\ref{E.1}) as well. Furthermore, let $q^{*}$ be the smallest (odd) value in a non-trivial cycle, should one exist. Then Proposition (\ref{P.7}) gives a rep for the cycle of the form
\begin{equation}\label{E:19}
\zeta_{s}(q^{*})=\left(2^{a_{k+1}}q^{*}-\sum\nolimits_{i=0}^{k}2^{a_{i}}3^{k-i}\right) / 3^{k+1}=q^{*},
\end{equation}
where $1<q^{*}\in \mathbb{N},k\in \mathbb{N}_{0}$, and $0= a_{0}<a_{1}<\cdots <a_{k}<a_{k+1}$. This formula is the same as Formula (7.3) in Crandall \cite{rC1978}, showing us that Proposition (\ref{P.7}) is equivalent to Theorem (7.1) in Crandall \cite{rC1978}.
\end{example}

Hence Proposition (\ref{P.3}) holds for Proposition (\ref{P.7}) for sequences which contain the trivial cycle, where $\left\Vert s\right\Vert =a_{k+1}$ and where the value $k$ carries the same meaning. Likewise, we have a rough comparison of Proposition (\ref{P.7}), rep (\ref{E.1}) and $E \cup O$ as follows
\begin{align}\notag
&\alpha_{k+1}+1=a_{k+1}-a_{k}=c_{k-1}+2b_{k}3^{k}\\ \notag
&\alpha_{k}+1=a_{k}-a_{k-1}=\upsilon_{k-1}+2b_{k-1}3^{k-1}\\ \notag
&\vdots\\ \notag
&\alpha_{2}+1=a_{2}-a_{1}=\upsilon_{1}+2b_{1}3^1\\ \notag
&\alpha_{1}+1=a_{1}-a_{0}=\upsilon_{0}+2b_{0}3^0. \notag
\end{align}

\begin{challenge}Derive a similar result to that of Theorem (\ref{T.3}) for Equation (\ref{E:19}) or generalize Theorem (\ref{T.3}).
\end{challenge}

Wirsching \cite{gW1996} gives the following definition of \textit{small sequences}.
\begin{definition}\label{D.9}
Denote by $L$ the set of finite sequences $\bigl\{(\alpha_{0},\dots,\alpha_{\mu})\in F\bigr\}$, where $\alpha_{i}<2\cdot3^{i-1}$ for $i=0,\dots,\mu$.
\end{definition} It should be clear that $L$ contains sequences which are not admissible (ones which do not evaluate to a natural number), and that it also contains admissible ones along with admissible duplicates per Proposition (\ref{P.2}). Note that
$$
\left\Vert s\right\Vert \leq \sum\nolimits_{i=1}^{\mu+1}\left(2\cdot3^{i-1}-1\right)+\mu+1=3^{\mu+1}-1,
$$ for $s \in L$. Clearly our definition of primitive seeds is essentially the same as the definition of small sequences since we fix $b_{i}=b_{0}=0$ to get
$$
\left\Vert s\right\Vert \leq c_{k}+\sum\nolimits_{i=1}^{k}\upsilon_{i}+\upsilon_{0}<3^{k+1}-1.$$
Note the subtle difference that our sum uses $<$ instead of $=$ as $E \cup O$ only enumerates every admissible sequence in $L$ (and thus every residue class), where it is understood that these sequences correlate with the ones containing the trivial cycle. It follows that $b_{i}$ and $b_{0}$ enumerate every positive-odd integer solution of every residue class, and therefore every odd natural number solution to rep (\ref{E.1}).
\begin{example}\label{X.11}It helps to recall the proof of corner case (\ref{E.5}) where we have 
$$
\sum\nolimits_{i=1}^{k}\upsilon_{i}+\upsilon_{0} =\sum\nolimits_{i=1}^{k}2\cdot3^{i}+2=3^{k+1}-1.
$$
where $\upsilon_{0}:=2$ as (\ref{E.5}) is in $O$. However, $c=3^{k+1}-1$ gives the solutions, not $c=2\cdot3^{k+1}$, and 
$$
3^{k+1}-1+3^{k+1}-1=2\cdot3^{k+1}-2<3^{k+2}-1.
$$
\end{example}

\vspace{2mm} \noindent \footnotesize
\begin{minipage}[b]{10cm}
Jeffrey R. GOODWIN, \\
Alumnus of St. Edward's University, \\
3001 South Congress Avenue, \\
Austin, Texas, 78704, USA. \\
Email: jeff.r.goodwin@austin.rr.com\\
Please contact the author for an electronic copy of his first paper.
\end{minipage}

\end{document}